\documentclass[11pt,reqno]{amsart}
\usepackage{amsfonts}
\usepackage{amssymb}
\usepackage{amscd}
\usepackage{hyperref}
\usepackage{color}

\usepackage{xcolor,cancel}
 
\def \bui#1#2{\mathrel{\mathop{\kern 0pt#1}\limits^{#2}}}

\newtheorem{theorem}{Theorem}[section]
\newtheorem{corollary}[theorem]{Corollary}
\newtheorem{lemma}[theorem]{Lemma}

\newtheorem{definition}[theorem]{Definition}
\newtheorem{remark}[theorem]{Remark}
\newtheorem{notation}[theorem]{Notation}

\numberwithin{equation}{section}

\begin{document}
\title{Scalar curvature under the collapse of metric}
\author[K.~Nguyen]{Khoi Nguyen}
\address{Department of Mathematics \\
Texas Christian University \\
Fort Worth, Texas 76129, USA}
\email[K.~Nguyen]{khoi.nguyen@tcu.edu}

\begin{abstract}
We prove a formula involving the scalar curvature of a Riemannian manifold endowed with a distribution in terms of an adapted orthonormal frame for its tangent bundle. Using the formula, we then investigate the effect of collapsing the metric along the distribution on the scalar curvature. This result contributes to the question of finding a positive scalar curvature metric on a Riemannian manifold.

\end{abstract}

\maketitle

\section{Introduction}
The problem of finding conditions on a Riemannian manifold so that it admits a positive scalar curvature (PSC) metric is a well-known open problem within the research community. The index of the Dirac operator on a spin manifold is an example of an obstruction to the existence of PSC metric; Rosenberg showed additional obstructions in (\cite{MR866507}). On the other hand, Gromov and Lawson (\cite{MR577131}) proved that the connected sum of two Riemannian manifolds, each admitting a PSC metric, admits a PSC metric as well. There are many similar results about those metrics, some of which are catalogued in Carlotto (\cite{MR4218620}).\\
\indent In this paper, we investigate the effect of collapsing a metric along a distribution on the scalar curvature. Specifically, it is well-known that if $M$ and $N$ are Riemannian manifolds and if $M$ admits a PSC metric, collapsing along $M$ in the product $M\times N$ produces a PSC metric of $M\times N$. We investigate a similar situation, but on a single manifold. Precisely, here is our formulation:\\
\indent Consider a $n$-dimensional Riemannian manifold $(M,g)$ and its tangent bundle $TM$. Suppose we can split $TM$ into two orthogonal sub-bundles $X$ and $Y$, with ranks $r$ and $s$ respectively. Declare $g'=g_X\oplus\left(\frac{1}{f^2}g_Y\right)$ to be a new metric on $M$, where $g_X$ is the metric restricted to $X$ and $g_Y$ is the metric restricted to $Y$. We then find a condition on the sub-bundles $X\,, Y$ and the constant $f$ so that if the scalar curvature of the manifold when restricted to one of the sub-bundles is positive, we can ensure that there is a PSC metric on the whole manifold.\\
\indent Let $(M,g)$ be a Riemannian manifold and $p\in M$. In a neighborhood $U$ of $p$ in $M$, we choose an adapted orthonormal frame of $M$ with respect to $g$. Denote that orthonormal frame by $\beta=\{e_1,e_2,\cdots,e_n\}$ and the corresponding dual coframe by $\beta'=\{e^1,e^2,\cdots,e^n\}$. \\
\indent The following definition is motivated by the definition of the structure constants of a Lie algebra.
\begin{definition}
The \textbf{structure functions} of $\beta$ are the functions $c^k_{ij}$ so that
    $$[e_i,e_j]=\sum\limits_kc^k_{ij}e_k.$$
\end{definition}
\indent Using the definition above, in this paper, we derived (see Theorem 2.4) a formula of the scalar curvature of $M$ entirely in terms of the structure functions of a local adapted orthonormal frame. We then derived (see Theorem 3.4) the formula relating the scalar curvature of $M$ and the scalar curvatures when restricted to the sub-bundles $X$ and $Y$.\\
\indent In Section 4.1, we provide some cases where our formula in Theorem 3.4 can be further simplified to aid with future computations should they arise. In Section 4.2, we investigate the effect of collapsing the metric along the sub-bundle $Y$. In the case where $Y$ is everywhere non-involutive (see Definition 4.4), collapsing along $Y$ will make the scalar curvature of $M$ decrease without bound (see Corollary 4.5). In the case that $Y$ is tangent bundle to a Riemannian foliation with leaves of PSC, then collapsing the metric along $Y$ causes the scalar curvature of $M$ to increase without bound (see Corollary 4.8). Utilizing Theorem 3.4, we are able to obtain the same result under weaker conditions, where the foliation need not be Riemannian. We call the condition on the foliation a nearly positively bundle-like metric (see Definition 4.6). With this condition, collapsing the metric along $Y$ causes the scalar curvature of $M$ to increase without bound (Theorem 4.7). In Section 5, we provide two examples to illustrate Corollary 4.5 and Theorem 4.7. \\
\indent This paper was part of my undergraduate Honors thesis at Texas Christian University under the guidance of Dr. Ken Richardson.
\section{A useful formula for the scalar curvature}
\indent In this section, we will derive a formula for the scalar curvature of a manifold in terms of the structure functions defined in Definition 1.1. But first, there are a few reminders about the quantities used in our computations below.\\
\indent The Christoffel symbols with respect to $\beta$ are the functions $\mathbf{\Gamma}^k_{ij}$ so that $\nabla_{e_i}e_j=\sum\limits_k\mathbf{\Gamma}^k_{ij} e_k$. They can be easily calculated (\cite{MR3887684}) via the formula $\mathbf{\Gamma}^k_{ij}=\frac{1}{2}(c^k_{ij}-c^i_{jk}+c^j_{ki})$. Next, the connection 1-form of a manifold $M$ with respect to the orthonormal frame $\beta$ can be defined as a matrix of 1-forms $\omega$ so that $\omega^j_{\hspace{0.5 em} i}=\sum\limits_p\mathbf{\Gamma}^j_{pi}e^p$. Finally, the curvature 2-form of $M$ is defined as $\Omega=d\omega+\omega\wedge\omega$. It is also well-known (\cite{MR1735502}) that $\Omega^k_{{\hspace{0.5 em}}j}=\sum\limits_{j,k}\frac{1}{2}R^k_{\hspace{0.5em}jiq}e^i\wedge e^q$, and thus the scalar curvature $S$ of $M$ can be calculated via the formula $S=\sum\limits_{j,k} \Omega^k_{\hspace{0.5em}j} (e_k,e_j)$.
\begin{notation}
Unless otherwise stated, we will use the following summation convention: Whenever there are repeated indices in the terms, those indices are implicitly summed over their appropriate range.
\end{notation}
\begin{lemma}
In terms of the structure functions and the dual co-frame, the curvature 2-form satisfies
    $$\Omega^k_{\hspace{0.5em}j}=\frac{1}{2}e_q(c^k_{ij}-c^i_{jk}+c^j_{ki})e^q\wedge e^i+\frac{1}{2}(c^k_{ij}-c^i_{jk}+c^j_{ki})de^i+\frac{1}{4}(c^k_{ml}-c^m_{lk}+c^l_{km})(c^l_{pj}-c^p_{jl}+c^j_{lp})e^m\wedge e^p.$$
\end{lemma}
\begin{proof}
By above, we have
    $$\omega^k_{\hspace{0.5em}j}=\mathbf{\Gamma}^k_{ij}e^i=\frac{1}{2}(c^k_{ij}-c^i_{jk}+c^j_{ki})e^i.$$
Now, using the definition of the curvature 2-form to write the entries explicitly, we have
\begin{eqnarray*}
    \Omega^k_{\hspace{0.5em}j}&=&d\omega^k_{\hspace{0.5em}j}+\omega^k_{\hspace{0.5em}l}\wedge\omega^l_{\hspace{0.5em}j}\\
    &=&\frac{1}{2}d[(c^k_{ij}-c^i_{jk}+c^j_{ki})e^i]+\frac{1}{4}[(c^k_{ml}-c^m_{lk}+c^l_{km})e^m]\wedge [(c^l_{pj}-c^p_{jl}+c^j_{lp})e^p]\\
    &=&\frac{1}{2}d(c^k_{ij}-c^i_{jk}+c^j_{ki})\wedge e^i+\frac{1}{2}(c^k_{ij}-c^i_{jk}+c^j_{ki})de^i\\
    \qquad &+&\frac{1}{4}(c^k_{ml}-c^m_{lk}+c^l_{km})(c^l_{pj}-c^p_{jl}+c^j_{lp})e^m\wedge e^p\\
    &=&\frac{1}{2}e_q(c^k_{ij}-c^i_{jk}+c^j_{ki})e^q\wedge e^i+\frac{1}{2}(c^k_{ij}-c^i_{jk}+c^j_{ki})de^i\\
    \qquad &+&\frac{1}{4}(c^k_{ml}-c^m_{lk}+c^l_{km})(c^l_{pj}-c^p_{jl}+c^j_{lp})e^m\wedge e^p,
\end{eqnarray*}
and the formula is proven.
\end{proof}
\begin{theorem} 
The formula for the scalar curvature $S$ of $M$ is
    $$S=2e_k(c^j_{kj})-c^k_{ik}c^j_{ij}-\frac{1}{2}c^i_{kj}c^k_{ij}-\frac{1}{4}c^i_{kj}c^i_{kj}.$$
\end{theorem}
\begin{proof}
\indent By the remarks before Notation 2.2, we have  $S=\Omega^k_{\hspace{0.5em}j} (e_k,e_j)$. Hence, plugging the formula for $\Omega$ as in Lemma 2.2, we have
\begin{eqnarray*}
    S&=&\Omega^k_{\hspace{0.5em}j}(e_k,e_j)\\
    &=&\frac{1}{2}e_q(c^k_{ij}-c^i_{jk}+c^j_{ki})e^q\wedge e^i(e_k,e_j)+\frac{1}{2}(c^k_{ij}-c^i_{jk}+c^j_{ki})de^i(e_k,e_j)\\
    \qquad &+&\frac{1}{4}(c^k_{ml}-c^m_{lk}+c^l_{km})(c^l_{pj}-c^p_{jl}+c^j_{lp})e^m\wedge e^p(e_k,e_j)\\
    &=&\frac{1}{2}e_q(c^k_{ij}-c^i_{jk}+c^j_{ki})(\delta_{qk}\delta_{ij}-\delta_{qj}\delta_{ik})+\frac{1}{2}(c^k_{ij}-c^i_{jk}+c^j_{ki})(-c^i_{kj})\\
    \qquad &+&\frac{1}{4}(c^k_{ml}-c^m_{lk}+c^l_{km})(c^l_{pj}-c^p_{jl}+c^j_{lp})(\delta_{mk}\delta_{pj}-\delta_{mj}\delta_{pk})\\
    &=&\frac{1}{2}e_k(c^k_{jj}-c^j_{jk}+c^j_{kj})-\frac{1}{2}e_j(c^k_{kj}-c^k_{jk}+c^j_{kk})-c^i_{kj}(c^k_{ij}-c^i_{jk}+c^j_{ki})\\
    \qquad &+&\frac{1}{4}(c^k_{kl}-c^k_{lk}+c^l_{kk})(c^l_{jj}-c^j_{jl}+c^j_{lj})-\frac{1}{4}(c^k_{jl}-c^j_{lk}+c^l_{kj})(c^l_{kj}-c^k_{jl}+c^j_{lk}),
\end{eqnarray*}
whereby we used various identities from exterior calculus and the Kronecker delta symbol to simplify. Continue using the antisymmetry of the structure functions and re-arranging the terms, we get
\begin{eqnarray*}
    S&=&\frac{1}{2}e_k(c^j_{kj}+c^j_{kj})-\frac{1}{2}e_j(c^k_{kj}+c^k_{kj})-\frac{1}{2}c^i_{kj}(c^k_{ij}-c^i_{jk}+c^j_{ki})\\
    \qquad &+&\frac{1}{4}(c^k_{kl}+c^k_{kl})(-c^j_{jl}-c^j_{jl})-\frac{1}{4}(c^k_{jl}-c^j_{lk}+c^l_{kj})(c^l_{kj}-c^k_{jl}+c^j_{lk})\\
    &=&e_k(c^j_{kj})-e_j(c^k_{kj})-\frac{1}{2}c^i_{kj}(c^k_{ij}-c^i_{jk}+c^j_{ki})\\
    \qquad& -&c^k_{kl}c^j_{jl}-\frac{1}{4}(c^k_{jl}c^l_{kj}-c^k_{jl}c^k_{jl}+c^k_{jl}c^j_{lk}-c^j_{lk}c^l_{kj}+c^j_{lk}c^k_{jl}-c^j_{lk}c^j_{lk})\\
    \qquad &-&\frac{1}{4}(c^l_{kj}c^l_{kj}-c^l_{kj}c^k_{jl}+c^l_{kj}c^j_{lk})\\
    &=&e_k(c^j_{kj})-e_j(c^k_{kj})-\frac{1}{2}c^i_{kj}(c^k_{ij}-c^i_{jk}+c^j_{ki})\\
    \qquad& -&c^k_{kl}c^j_{jl}+\frac{1}{4}(c^k_{jl}c^k_{jl}+c^j_{lk}c^j_{lk}-c^l_{kj}c^l_{kj})+c^k_{lj}c^j_{lk}\\
    &=&e_k(c^j_{kj})-e_j(c^k_{kj})+\frac{1}{4}(c^k_{jl}c^k_{jl}+c^j_{lk}c^j_{lk}-c^l_{kj}c^l_{kj})\\
    \qquad &+&\frac{1}{2}c^k_{lj}c^j_{lk}-c^k_{kl}c^j_{jl}-\frac{1}{2}c^i_{kj}c^k_{ij}-\frac{1}{2}c^i_{kj}c^i_{kj}-\frac{1}{2}c^i_{jk}c^j_{ik}\\
    &=&e_k(c^j_{kj})-e_j(c^k_{kj})+\frac{1}{4}c^k_{jl}c^k_{jl}+\frac{1}{4}c^j_{lk}c^j_{lk}-\frac{1}{4}c^l_{kj}c^l_{kj}\\
    \qquad &+&\frac{1}{2}c^k_{lj}c^j_{lk}-\frac{1}{2}c^i_{jk}c^j_{ik}-c^k_{kl}c^j_{jl}-\frac{1}{2}c^i_{kj}c^k_{ij}-\frac{1}{2}c^i_{kj}c^i_{kj}.
\end{eqnarray*}
\indent Since $i,j,k,l$ are dummy indices, we can change the index in each term and simplify more. After a series of index changing and simplifying, we get
\begin{eqnarray*}
    S&=&2e_k(c^j_{kj})+\frac{1}{4}c^i_{kj}c^i_{kj}-c^k_{kl}c^j_{jl}-\frac{1}{2}c^i_{kj}c^k_{ij}-\frac{1}{2}c^i_{kj}c^i_{kj}\\
    &=&2e_k(c^j_{kj})-c^k_{kl}c^j_{jl}-\frac{1}{2}c^i_{kj}c^k_{ij}-\frac{1}{4}c^i_{kj}c^i_{kj}.
\end{eqnarray*}
\indent Finally, changing $l$ to $i$ in the second term, we arrive at the desired formula for $S$. 
\end{proof}
\begin{remark}
The formula in Theorem 2.3 is consistent with the known formula of the scalar curvature of a Lie group (see Section 4.1 below and \cite{MR425012}, Lemma 1.1).
\end{remark}
\indent As it turns out, this formula will be very important in our later sections to investigate the case of positive scalar curvature metrics on a Riemannian manifold. Also, in order to be clear about the terms, we may put back the summation notation and get
\begin{eqnarray}
S=\sum\limits_{k,j=1}^n 2e_k(c^j_{kj})+\sum\limits_{i,j,k=1}^n \left[-c^k_{ik}c^j_{ij}- \frac{1}{2}c^i_{kj}c^k_{ij}-\frac{1}{4}(c^i_{kj})^2\right]
\end{eqnarray}

\section{The decomposition of the scalar curvature in terms of the sub-bundles}
With the formula for the scalar curvature in terms of the structure functions proven, in this section, we begin to investigate the situation laid out at the end of Section 1. But first, we will re-formulate the scenario we are investigating with some additional notation.\\
\indent We start with a $n$-dimensional Riemannian manifold $(M,g)$ and its tangent bundle $TM$, and suppose we can split $TM$ into two orthogonal sub-bundles $X$ and $Y$, with ranks $r$ and $s$ respectively. In other words, $TM=X\oplus Y$ and $r+s=n$. Declare $g'=g_X\oplus\left(\frac{1}{f^2}g_Y\right)$ to be a new metric on $M$, where $g_X$ is the metric restricted to $X$, $g_Y$ is the metric restricted to $Y$ and $f$ is a positive constant. Also, let $\beta=\{e_1,e_2,\cdots, e_{r},e_{r+1},\cdots, e_n\}$ be an adapted local orthonormal frame with respect to $g$. Denote $S$ the scalar curvature of $M$ with respect to $g'$, $S_1$ the scalar curvature of $M$ restricted to $X$ and $S_2$ the scalar curvature of $M$ restricted to $Y$ (both with respect to $g$).
\begin{notation}
In the calculations below, denote the indices corresponding to the first $r$-coordinates by lowercase Roman letters, the indices corresponding to the last $s$ coordinates by Greek letters, and the indices corresponding to generic coordinates by capital Roman letters.
\end{notation}
\indent In order to extract useful information to our analysis, we may want to find a relationship between the scalar curvatures $S\,,S_1\,,S_2$ and the constant $f$, so that, for instance, if we know that $S_2>0$, what can we conclude about the sign of $S$?\\
\indent First, by linear algebra, it can be proven that with the new metric $g'$ defined as above, if $\beta'=\{e'_i:1\leq i\leq n\}$ is an adapted orthonormal frame with respect to $g'$, then $e'_i=Pe_i$ for all $1\leq i\leq n$, where $P=I_{r\times r}\oplus \left(fI\right)_{s\times s}$. In other words, $e'_a=e_a$ for all $1\leq a\leq r$ and $e'_{\alpha}=fe_{\alpha}$ for all $r+1\leq \alpha\leq n$, using Notation 3.1.
\begin{lemma}
Let $c^P_{{MQ}}$ and $\widetilde{c^P_{MQ}}$ be the structure functions of $\beta$ and $\beta'$ respectively. If, in the form of $P$ as above, $f$ is a function that changes according to the location on $M$, then we have the following relationship, written using Notation 3.1:
\begin{eqnarray*}
    \widetilde{c^k_{ij}}&=&c^k_{ij},\\
    \widetilde{c^{\alpha}_{ij}}&=&\frac{1}{f}c^{\alpha}_{ij},\\
     \widetilde{c^{k}_{i\alpha}}&=&fc^k_{i\alpha},\\
       \widetilde{c^{\beta}_{i\beta}}&=&\frac{1}{f}e_i(f)+c^{\beta}_{i\beta},\\
    \widetilde{c^{\beta}_{i\alpha}}&=&c^{\beta}_{i\alpha}\, (\text{when $\beta\neq\alpha$}),\\
     \widetilde{c^k_{\alpha\beta}}&=&f^2c^k_{\alpha\beta},\\
       \widetilde{c^{\gamma}_{\gamma\beta}}&=&f c^{\gamma}_{\gamma\beta}-e_{\beta}(f)\,(\text{when $\gamma\neq\beta$}),\\
    \widetilde{c^{\gamma}_{\alpha\beta}}&=&fc^{\gamma}_{\alpha\beta}.
\end{eqnarray*}
\end{lemma}
\begin{proof}
By the remark above, we have the following change in the frame field vectors:
\begin{eqnarray*}
    e'_i&=&e_i\\
    e'_{\alpha}&=&fe_{\alpha}
\end{eqnarray*}
\indent There are 8 cases to consider, depending on the values of the indices of the structure functions. ($M,P,Q$ can be either between $1$ and $r$ or between $r+1$ and $n$). We will carry out the computations for the first 4 cases, and the remaining cases are similar.\\
\indent \textbf{Case 1}: If $1\leq M,Q,P\leq r$, then in this case, we may use Notation 3.1 to write $M=i\,,Q=j\,,P=k$. Hence, in this case, $[e'_i,e'_j]=\widetilde{c^k_{ij}}e'_k$, while $[e_i,e_j]=c^k_{ij}e_k$. However, since $1\leq i,j\leq r$, $[e'_i,e'_j]=[e_i,e_j]$ and $e'_k=e_k$, so 
\begin{eqnarray}
    \widetilde{c^k_{ij}}e_k&=&c^k_{ij}e_k.
\end{eqnarray}
\indent \textbf{Case 2}: If $1\leq M,Q\leq r,r+1\leq P\leq n$, then in this case, we may use Notation 3.1 to write $M=i\,,Q=j\,,P=\alpha$. Hence, in this case,  $[e'_i,e'_j]=\widetilde{c^{\alpha}_{ij}}e'_\alpha=\widetilde{c^{\alpha}_{ij}}fe_{\alpha}$. On the other hand, $[e_i,e_j]=c^{\alpha}_{ij}e_{\alpha}$. However, by the same reasoning as above, we get
\begin{eqnarray}
    f\widetilde{c^{\alpha}_{ij}}e_{\alpha}&=&c^{\alpha}_{ij}e_{\alpha}.
\end{eqnarray}
\indent \textbf{Case 3}: If $1\leq M,P\leq r,r+1\leq Q\leq n$ then in this case, we may use Notation 3.1 to write $M=i\,,Q=\alpha \,,P=k$. Hence, in this case, $[e'_i,e'_{\alpha}]=\widetilde{c^k_{i\alpha}}e'_k=\widetilde{c^k_{i\alpha}}e_k$. On the other hand,
\begin{eqnarray*}
    [e'_i,e'_{\alpha}]&=&[e_i,fe_{\alpha}]\\
    &=&e_i(fe_{\alpha})-fe_{\alpha}e_i\\
    &=&e_i(f)e_{\alpha}+fe_ie_{\alpha}-fe_{\alpha}e_i\\
    &=&e_i(f)e_{\alpha}+f[e_i,e_{\alpha}]\\
    &=&e_i(f)e_{\alpha}+fc^k_{i\alpha}e_k.
\end{eqnarray*}
\indent Hence,
\begin{eqnarray}
    \widetilde{c^k_{i\alpha}}e_k&=&e_i(f)e_{\alpha}+fc^k_{i\alpha}e_k.
\end{eqnarray}
\indent \textbf{Case 4}: If $1\leq M\leq r,r+1\leq Q,P\leq n$ then in this case, we may use Notation 3.1 to write $M=i\,,Q=\alpha \,,P=\beta$. Hence, in this case, $[e'_i,e'_{\alpha}]=\widetilde{c^{\beta}_{i\alpha}}e'_{\beta}=\widetilde{c^{\beta}_{i\alpha}}fe_{\beta}$. On the other hand,
\begin{eqnarray*}
    [e'_i,e'_{\alpha}]&=&[e_i,fe_{\alpha}]\\
    &=&e_i(fe_{\alpha})-fe_{\alpha}e_i\\
    &=&e_i(f)e_{\alpha}+fe_ie_{\alpha}-fe_{\alpha}e_i\\
    &=&e_i(f)e_{\alpha}+f[e_i,e_{\alpha}]\\
    &=&e_i(f)e_{\alpha}+fc^{\beta}_{i\alpha}e_{\beta}.
\end{eqnarray*}
\indent Hence,
\begin{eqnarray}
    \widetilde{c^{\beta}_{i\alpha}}fe_{\beta}&=&e_i(f)e_{\alpha}+fc^{\beta}_{i\alpha}e_{\beta}.
\end{eqnarray}
\indent Now, notice that in Case 1 and 2, the expansion of $[e'_i,e'_j]$ must be the same on both sides of the equality. Hence, we must have matching coefficients when adding (3.1) and (3.2). Therefore, matching the coefficients of $e_k$ and $e_{\alpha}$ respectively in (3.1) and (3.2) yields 
\begin{eqnarray}
    \widetilde{c^k_{ij}}&=&c^k_{ij}\\
    \widetilde{c^{\alpha}_{ij}}&=&\frac{1}{f}c^{\alpha}_{ij}.
\end{eqnarray}
\indent A similar argument as above applied to (3.3) and (3.4) yields
\begin{eqnarray*}
    \widetilde{c^{k}_{i\alpha}}&=&fc^k_{i\alpha}.
\end{eqnarray*}
\indent and another pair of equations. Notice that if $\alpha=\beta$, then (3.4) will yield $f\widetilde{c^{\beta}_{i\beta}}=e_i(f)+fc^{\beta}_{i\beta}$, while if $\beta\neq\alpha$, then the contribution of $e_{\alpha}$ is negligible, and so $f\widetilde{c^{\beta}_{i\alpha}}=fc^{\beta}_{i\alpha}$. Hence, combining the cases, we have
\begin{eqnarray*}
    \widetilde{c^{\beta}_{i\beta}}&=&\frac{1}{f}e_i(f)+c^{\beta}_{i\beta},\\
    \widetilde{c^{\beta}_{i\alpha}}&=&c^{\beta}_{i\alpha} (\text{when $\beta\neq\alpha$}).
\end{eqnarray*}
\indent The other cases are similar in the calculations. Combining all the cases, Lemma 3.1 is proven.
\end{proof}
\indent In the case that $f$ is a constant function, an immediate corollary to Lemma 3.1 is
\begin{corollary}
\indent If $f$ is a constant, then the structure functions will undergo the following changes.
\begin{eqnarray*}
    \widetilde{c^k_{ij}}&=&c^k_{ij}\\
    \widetilde{c^{\alpha}_{ij}}&=&\frac{1}{f}c^{\alpha}_{ij}\\
     \widetilde{c^{k}_{i\alpha}}&=&fc^k_{i\alpha}\\
       \widetilde{c^{\beta}_{i\beta}}&=&c^{\beta}_{i\beta}\\
    \widetilde{c^{\beta}_{i\alpha}}&=&c^{\beta}_{i\alpha}\,(\text{when $\beta\neq\alpha$})\\
     \widetilde{c^k_{\alpha\beta}}&=&f^2c^k_{\alpha\beta}\\
       \widetilde{c^{\gamma}_{\gamma\beta}}&=&f c^{\gamma}_{\gamma\beta}\,(\text{when $\gamma\neq\beta$})\\
    \widetilde{c^{\gamma}_{\alpha\beta}}&=&fc^{\gamma}_{\alpha\beta}
\end{eqnarray*}
\end{corollary}
\indent Having done the preparatory work, we will prove the most important theorem of this paper.
\begin{theorem}
In the set up as in the introductory paragraphs to this section, we have the following relationship, where the repeated indices are summed over their appropriate ranges:
\begin{eqnarray}
S&=&S_1+f^2S_2-\frac{1}{4}f^4c^i_{\gamma\beta}c^i_{\gamma\beta}\nonumber\\
    \qquad&+&\left(2e_{\gamma}(c^j_{\gamma j})-c^k_{\alpha k}c^j_{\alpha j}-2c^k_{\alpha k}c^{\beta}_{\alpha\beta}-\frac{1}{2}c^k_{i\beta}c^i_{k\beta}-c^i_{\gamma\beta}c^{\gamma}_{i\beta}-\frac{1}{2}c^i_{\gamma j}c^i_{\gamma j}\right)f^2\nonumber\\
    \qquad&+&\left(2e_k(c^{\beta}_{k\beta})-c^{\gamma}_{i\gamma}c^{\beta}_{i\beta}-2c^{\gamma}_{i\gamma}c^j_{ij}-\frac{1}{2}c^{\alpha}_{\gamma j}c^{\gamma}_{\alpha j}-c^i_{\gamma j}c^{\gamma}_{ij}-\frac{1}{2}c^{\alpha}_{k\beta}c^{\alpha}_{k\beta}\right)\nonumber\\
    \qquad&-&\frac{1}{4}\left(\frac{1}{f^2}\right)(c^{\alpha}_{kj}c^{\alpha}_{kj}).
\end{eqnarray}
\end{theorem}
\begin{proof}
Notice that by the formula at the end of the last section, we have
\begin{eqnarray*}
    S&=&\sum\limits_{J,K=1}^n 2e_K(c^J_{KJ})+\sum\limits_{I,J,K=1}^n \left[-c^K_{IK}c^J_{IJ}-\frac{1}{2}c^I_{KJ}c^K_{IJ}-\frac{1}{4}(c^I_{KJ})^2\right].
\end{eqnarray*}
\indent Using Notation 3.1, we can decompose $S$ as
\begin{eqnarray*}
    S&=&\sum\limits_{j,k=1}^r 2e_k(c^j_{kj})+\sum\limits_{j,k=1}^r\left[-c^k_{ik}c^j_{ij}-\frac{1}{2}c^i_{kj}c^k_{ij}-\frac{1}{4}(c^i_{kj})^2\right]\\
    \qquad&+&\sum\limits_{\beta,\gamma=r+1}^{r+s} 2e_{\gamma}(c^{\beta}_{\gamma\beta})+\sum\limits_{\alpha,\beta,\gamma=r+1}^{r+s}\left[-c^{\gamma}_{\alpha\gamma}c^{\beta}_{\alpha\beta}-\frac{1}{2}c^{\alpha}_{\gamma\beta}c^{\gamma}_{\alpha\beta}-\frac{1}{4}(c^{\alpha}_{\gamma\beta})^2\right]\\
    \qquad&+&\sum\limits_{i=1}^r\sum\limits_{\gamma,\beta=r+1}^{r+s}\left[-c^{\gamma}_{i\gamma}c^{\beta}_{i\beta}-\frac{1}{2}c^i_{\gamma\beta}c^{\gamma}_{i\beta}-\frac{1}{4}(c^i_{\gamma\beta})^2\right]\\
    \qquad&+&\sum\limits_{j=1}^r\sum\limits_{\gamma=r+1}^{r+s}2e_{\gamma}(c^j_{\gamma j})+\sum\limits_{i,j=1}^r\sum\limits_{\gamma=r+1}^{r+s}\left[-c^{\gamma}_{i\gamma}c^j_{ij}-\frac{1}{2}c^i_{\gamma j }c^{\gamma}_{ij}-\frac{1}{4}(c^i_{\gamma j})^2\right]\\
    \qquad&+&\sum\limits_{k=1}^r\sum\limits_{\beta=r+1}^{r+s}2e_k(c^{\beta}_{k\beta})+\sum\limits_{i,k=1}^r\sum\limits_{\beta=r+1}^{r+s}\left[c^k_{ik}c^{\beta}_{i\beta}-\frac{1}{2}c^i_{k\beta}c^k_{i\beta}-\frac{1}{4}(c^i_{k\beta})^2\right]\\
    \qquad&+&\sum\limits_{j=1}^r\sum\limits_{\alpha,\gamma=r+1}^{r+s}\left[-c^\gamma_{\alpha\gamma}c^j_{\alpha j}-\frac{1}{2}c^{\alpha}_{\gamma j}c^{\gamma}_{\alpha j}-\frac{1}{4}(c^{\alpha}_{\gamma j})^2\right]\\
    \qquad&+&\sum\limits_{j,k=1}^r\sum\limits_{\alpha=r+1}^{r+s}\left[-c^k_{\alpha k}c^j_{\alpha j}-\frac{1}{2}c^{\alpha}_{kj}c^k_{\alpha j}-\frac{1}{4}(c^{\alpha}_{kj})^2\right]\\
    \qquad&+&\sum\limits_{k=1}^r\sum\limits_{\alpha,\beta=r+1}^{r+s}\left[-c^k_{\alpha k}c^{\beta}_{\alpha\beta}-\frac{1}{2}c^{\alpha}_{k\beta}c^k_{\alpha\beta}-\frac{1}{4}(c^{\alpha}_{k\beta})^2\right].
\end{eqnarray*}
\indent Taking a closer look at the terms containing both the Greek and Roman indices above, which we will call $Q$, we realize that we can simplify them. Specifically, rearranging and re-indexing the Roman indices with the Roman indices and the Greek indices with the Greek indices, we eventually get
\begin{eqnarray*}
    Q&=&2e_{\gamma}(c^j_{\gamma j})+2e_k(c^{\beta}_{k\beta})\\
    \qquad &-&c^{\gamma}_{i\gamma}c^{\beta}_{i\beta}-c^k_{\alpha k}c^j_{\alpha j}-2c^{\gamma}_{i\gamma}c^j_{ij}-2c^k_{\alpha k}c^{\beta}_{\alpha\beta}\\
    \qquad &-&\frac{1}{2}c^i_{k\beta}c^k_{i\beta}-\frac{1}{2}c^{\alpha}_{\gamma j}c^{\gamma}_{\alpha j}-c^i_{\gamma\beta}c^{\gamma}_{i\beta}-c^i_{\gamma j }c^{\gamma}_{ij}\\
    \qquad&-&\frac{1}{4}c^i_{\gamma\beta}c^i_{\gamma\beta}-\frac{1}{4}c^{\alpha}_{kj}c^{\alpha}_{kj}-\frac{1}{2}c^i_{\gamma j}c^i_{\gamma j}-\frac{1}{2}c^{\alpha}_{k\beta}c^{\alpha}_{k\beta},
\end{eqnarray*}
and so, the formula for the scalar curvature $S_0$ in terms of the components, after a few rearrangements, will be
\begin{eqnarray*}
    S_0&=&2e_k(c^j_{kj})+2e_{\gamma}(c^{\beta}_{\gamma\beta})+2e_{\gamma}(c^j_{\gamma j})+2e_k(c^{\beta}_{k\beta})\\
    &-&c^k_{ik}c^j_{ij}-\frac{1}{2}c^i_{kj}c^k_{ij}-\frac{1}{4}c^i_{kj}c^i_{kj}\\
    &-&c^{\gamma}_{\alpha\gamma}c^{\beta}_{\alpha\beta}-\frac{1}{2}c^{\alpha}_{\gamma\beta}c^{\gamma}_{\alpha\beta}-\frac{1}{4}c^{\alpha}_{\gamma\beta}c^{\alpha}_{\gamma\beta}\\
     \qquad &-&c^{\gamma}_{i\gamma}c^{\beta}_{i\beta}-c^k_{\alpha k}c^j_{\alpha j}-2c^{\gamma}_{i\gamma}c^j_{ij}-2c^k_{\alpha k}c^{\beta}_{\alpha\beta}\\
    \qquad &-&\frac{1}{2}c^i_{k\beta}c^k_{i\beta}-\frac{1}{2}c^{\alpha}_{\gamma j}c^{\gamma}_{\alpha j}-c^i_{\gamma\beta}c^{\gamma}_{i\beta}-c^i_{\gamma j }c^{\gamma}_{ij}\\
    \qquad&-&\frac{1}{4}c^i_{\gamma\beta}c^i_{\gamma\beta}-\frac{1}{4}c^{\alpha}_{kj}c^{\alpha}_{kj}-\frac{1}{2}c^i_{\gamma j}c^i_{\gamma j}-\frac{1}{2}c^{\alpha}_{k\beta}c^{\alpha}_{k\beta}.
\end{eqnarray*}
\indent Now, we will introduce the change in the metric to the scalar curvature. Notice that the new scalar curvature $S$ of $M$ will have formula
\begin{eqnarray*}
    S&=&2\widetilde{e_k}(\widetilde{c^j_{kj}})+2\widetilde{e_{\gamma}}(\widetilde{c^{\beta}_{\gamma\beta}})+2\widetilde{e_{\gamma}}(\widetilde{c^j_{\gamma j}})+2\widetilde{e_k}(\widetilde{c^{\beta}_{k\beta}})\\
    &-&\widetilde{c^k_{ik}}\widetilde{c^j_{ij}}-\frac{1}{2}\widetilde{c^i_{kj}}\widetilde{c^k_{ij}}-\frac{1}{2}\widetilde{c^i_{kj}}\widetilde{c^i_{kj}}\\
    &-&\widetilde{c^{\gamma}_{\alpha\gamma}}\widetilde{c^{\beta}_{\alpha\beta}}-\frac{1}{2}\widetilde{c^{\alpha}_{\gamma\beta}}\widetilde{c^{\gamma}_{\alpha\beta}}-\frac{1}{4}(\widetilde{c^{\alpha}_{\gamma\beta}})^2\\
     \qquad &-&\widetilde{c^{\gamma}_{i\gamma}}\widetilde{c^{\beta}_{i\beta}}-\widetilde{c^k_{\alpha k}}\widetilde{c^j_{\alpha j}}-2\widetilde{c^{\gamma}_{i\gamma}}\widetilde{c^j_{ij}}-2\widetilde{c^k_{\alpha k}}\widetilde{c^{\beta}_{\alpha\beta}}\\
    \qquad &-&\frac{1}{2}\widetilde{c^i_{k\beta}}\widetilde{c^k_{i\beta}}-\frac{1}{2}\widetilde{c^{\alpha}_{\gamma j}}\widetilde{c^{\gamma}_{\alpha j}}-\widetilde{c^i_{\gamma\beta}}\widetilde{c^{\gamma}_{i\beta}}-\widetilde{c^i_{\gamma j }}\widetilde{c^{\gamma}_{ij}}\\
    \qquad&-&\frac{1}{4}\widetilde{c^i_{\gamma\beta}}\widetilde{c^i_{\gamma\beta}}-\frac{1}{4}\widetilde{c^{\alpha}_{kj}}\widetilde{c^{\alpha}_{kj}}-\frac{1}{2}\widetilde{c^i_{\gamma j}}\widetilde{c^i_{\gamma j}}-\frac{1}{2}\widetilde{c^{\alpha}_{k\beta}}\widetilde{c^{\alpha}_{k\beta}}.
\end{eqnarray*}
\indent Using Corollary 3.3 to write the new structure functions back in terms of the old ones and rearranging, we get
\begin{eqnarray*}
    S&=&2e_k(c^k_{kj})-c^k_{ik}c^j_{ij}-\frac{1}{2}c^i_{kj}c^k_{ij}-\frac{1}{4}c^i_{kj}c^i_{kj}\\
    \qquad&+&2f^2e_{\gamma}(c^{\beta}_{\gamma\beta})-f^2c^{\gamma}_{\alpha\gamma}c^{\beta}_{\alpha\beta}-\frac{1}{2}f^2c^{\alpha}_{\gamma\beta}c^{\gamma}_{\alpha\beta}-\frac{1}{4}f^2c^{\alpha}_{\gamma\beta}c^{\alpha}_{\gamma\beta}\\
    \qquad&-&\frac{1}{4}f^4c^i_{\gamma\beta}c^i_{\gamma\beta}\\
    \qquad&+&\left(2e_{\gamma}(c^j_{\gamma j})-c^k_{\alpha k}c^j_{\alpha j}-2c^k_{\alpha k}c^{\beta}_{\alpha\beta}-\frac{1}{2}c^k_{i\beta}c^i_{k\beta}-c^i_{\gamma\beta}c^{\gamma}_{i\beta}-\frac{1}{2}c^i_{\gamma j}c^i_{\gamma j}\right)f^2\\
    \qquad&+&\left(2e_k(c^{\beta}_{k\beta})-c^{\gamma}_{i\gamma}c^{\beta}_{i\beta}-2c^{\gamma}_{i\gamma}c^j_{ij}-\frac{1}{2}c^{\alpha}_{\gamma j}c^{\gamma}_{\alpha j}-c^i_{\gamma j}c^{\gamma}_{ij}-\frac{1}{2}c^{\alpha}_{k\beta}c^{\alpha}_{k\beta}\right)\\
    \qquad&-&\frac{1}{4}\left(\frac{1}{f^2}\right)c^{\alpha}_{kj}c^{\alpha}_{kj}\\
    &=&S_1+f^2S_2-\frac{1}{4}f^4c^i_{\gamma\beta}c^i_{\gamma\beta}\\
    \qquad&+&\left(2e_{\gamma}(c^j_{\gamma j})-c^k_{\alpha k}c^j_{\alpha j}-2c^k_{\alpha k}c^{\beta}_{\alpha\beta}-\frac{1}{2}c^k_{i\beta}c^i_{k\beta}-c^i_{\gamma\beta}c^{\gamma}_{i\beta}-\frac{1}{2}c^i_{\gamma j}c^i_{\gamma j}\right)f^2\\
    \qquad&+&\left(2e_k(c^{\beta}_{k\beta})-c^{\gamma}_{i\gamma}c^{\beta}_{i\beta}-2c^{\gamma}_{i\gamma}c^j_{ij}-\frac{1}{2}c^{\alpha}_{\gamma j}c^{\gamma}_{\alpha j}-c^i_{\gamma j}c^{\gamma}_{ij}-\frac{1}{2}c^{\alpha}_{k\beta}c^{\alpha}_{k\beta}\right)\\
    \qquad&-&\frac{1}{4}\left(\frac{1}{f^2}\right)c^{\alpha}_{kj}c^{\alpha}_{kj}.
\end{eqnarray*}
\indent Therefore, Theorem 3.4 is proven.
\end{proof}

\section{Special cases of the formula}
In this section, we will use formula (3.7) to analyze the condition on a Riemannian manifold and its sub-bundles as well as the constant $f$ in the formula of the new metric $g'$ so that $g'$ will be a PSC metric on our manifold. We will also provide special cases to Formula (3.7) that will help significantly with calculations.
\subsection{Simplifications of Theorem 3.4}
First, if the manifold $M$ happens to be a Lie group with a bi-invariant metric, then the structure functions now become constant, and Formula (2.1) becomes
\begin{eqnarray}
    S=\sum\limits_{i,j,k=1}^n \left[-c^k_{ik}c^j_{ij}- \frac{1}{2}c^i_{kj}c^k_{ij}-\frac{1}{4}(c^i_{kj})^2\right]
\end{eqnarray}
\indent In terms of our formula in Theorem 3.4, the resulting formula is
\begin{eqnarray*}
S&=&S_1+f^2S_2-\frac{1}{4}f^4c^i_{\gamma\beta}c^i_{\gamma\beta}\nonumber\\
    \qquad&+&\left(-c^k_{\alpha k}c^j_{\alpha j}-\frac{1}{2}c^k_{i\beta}c^i_{k\beta}-\frac{1}{2}c^i_{\gamma j}c^i_{\gamma j}\right)f^2\nonumber\\
    \qquad&+&\left(-c^{\gamma}_{i\gamma}c^{\beta}_{i\beta}-2c^{\gamma}_{i\gamma}c^j_{ij}-\frac{1}{2}c^{\alpha}_{\gamma j}c^{\gamma}_{\alpha j}-c^i_{\gamma j}c^{\gamma}_{ij}-\frac{1}{2}c^{\alpha}_{k\beta}c^{\alpha}_{k\beta}\right)\nonumber\\
    \qquad&-&\frac{1}{4}\left(\frac{1}{f^2}\right)(c^{\alpha}_{kj}c^{\alpha}_{kj}).
\end{eqnarray*}
\indent It should be noted that it is known (for instance, see \cite{MR425012}, Lemma 1.1) that in the case of a Lie group, the scalar curvature can be calculated via the structure constants of the underlying Lie algebra, so we can think of Formula 2.1 as a generalization of this case to arbitrary Riemannian manifolds.\\
\indent Next, if the codimension of $Y$ is 1 (in other words, $X$ is a 1-dimensional sub-bundle indexed by 1), then $S_1=0$ and Formula (3.7) reduces to
\begin{eqnarray*}
S&=&f^2S_2\nonumber\\
    \qquad&+&\left(2e_{\gamma}(c^1_{\gamma 1})-c^1_{\alpha 1}c^1_{\alpha 1}-\frac{1}{2}c^1_{1\beta}c^1_{1\beta}-\frac{1}{2}c^1_{\gamma 1}c^1_{\gamma 1}\right)f^2\nonumber\\
    \qquad&+&\left(2e_1(c^{\beta}_{1\beta})-c^{\gamma}_{1\gamma}c^{\beta}_{1\beta}-\frac{1}{2}c^{\alpha}_{\gamma 1}c^{\gamma}_{\alpha 1}-\frac{1}{2}c^{\alpha}_{1\beta}c^{\alpha}_{1\beta}\right)\nonumber\\
    &=&f^2S_2\nonumber\\
    \qquad&+&\left(2e_{\gamma}(c^1_{\gamma 1})-2c^1_{\alpha 1}c^1_{\alpha 1}\right)f^2\nonumber\\
    \qquad&+&\left(2e_1(c^{\beta}_{1\beta})-c^{\gamma}_{1\gamma}c^{\beta}_{1\beta}-\frac{1}{2}c^{\alpha}_{\gamma 1}c^{\gamma}_{\alpha 1}-\frac{1}{2}c^{\alpha}_{1\beta}c^{\alpha}_{1\beta}\right),\nonumber
\end{eqnarray*}
where the Greek indices run from 2, 3, etc.\\
\indent From Frobenius' Theorem, a sub-bundle $Y$ to the tangent bundle $TM$ of a manifold $M$ is involutive if and only if $Y$ is also a tangent space to a foliation $\mathcal{F}$ of $M$.
\begin{theorem}[Frobenius]
Let $(M,g)$ be a Riemannian manifold, and $B=\{e_1,\cdots, e_r,e_{r+1},\cdots e_n\}$ be a local adapted orthonormal frame. Then, using the notation below, a sub-bundle $Y$ of $M$ is involutive if and only if $c^i_{\alpha\beta}=0$ for all $\alpha\,,\beta\,, i$.
\end{theorem}
\indent A special type of involutive sub-bundle of a manifold is the one where the metric is bundle-like with respect to the sub-bundle.
\begin{definition}[Section IV, Proposition 4.2 in \cite{MR705126}]
Let $(M,g)$ be a Riemannian manifold and $Y$ be an involutive sub-bundle of $M$. Let $X=Y^{\perp}$. Then we say the metric $g$ is \textbf{bundle-like} with respect to $Y$ if for all $x\in M$, there exists a local adapted orthonormal frame $B=\{e_1,\cdots, e_r,e_{r+1},\cdots,e_n\}$ defined on an open neighborhood $U$ of $x$ so that $X|_U=\text{span}\{e_1,\cdots, e_r\}_U\,, Y|_U=\text{span}\{e_{r+1},\cdots, e_n\}_U$, and for all $\alpha, j$, $[e_{\alpha},e_j]\in\Gamma(U,Y)$.
\end{definition}
\indent As a consequence, we have
\begin{corollary}
Let $(M,g)$ be a Riemannian manifold and $Y$ be an involutive sub-bundle of $M$. Let $X=Y^{\perp}$. Then  $g$ is bundle-like with respect to $Y$ if and only if for all $x\in M$ there exists a local adapted orthonormal frame $B=\{e_1,\cdots, e_r,e_{r+1},\cdots,e_n\}$ in an open neighborhood of $x$ so that for all $\alpha, j,k$, $c^k_{\alpha j}=0$.
\end{corollary}
\indent If $Y$ is a 1-dimensional foliation, indexed by 1, then $S_2=0$, and by Theorem 4.3, the structure functions $c^i_{\alpha\beta}=0$. Formula (3.7) reduces to
\begin{eqnarray*}
S&=&S_1\nonumber\\
    \qquad&+&\left(2e_1(c^j_{1 j})-c^k_{1 k}c^j_{1 j}-\frac{1}{2}c^k_{i1}c^i_{k1}-\frac{1}{2}c^i_{1 j}c^i_{1 j}\right)f^2\nonumber\\
    \qquad&+&\left(2e_k(c^{1}_{k1})-c^1_{i1}c^{1}_{i1}-2c^{1}_{i1}c^j_{ij}-\frac{1}{2}c^{1}_{1 j}c^{1}_{1 j}-c^i_{1 j}c^{1}_{ij}-\frac{1}{2}c^{1}_{k1}c^{1}_{k1}\right)\nonumber\\
    \qquad&-&\frac{1}{4}\left(\frac{1}{f^2}\right)(c^{1}_{kj}c^{1}_{kj})\\
    &=&S_1\nonumber\\
    \qquad&+&\left(2e_1(c^j_{1 j})-c^k_{1 k}c^j_{1 j}-\frac{1}{2}c^k_{i1}c^i_{k1}-\frac{1}{2}c^i_{1 j}c^i_{1 j}\right)f^2\nonumber\\
    \qquad&+&\left(2e_k(c^{1}_{k1})-2c^1_{i1}c^1_{i1}-2c^{1}_{i1}c^j_{ij}-c^i_{1 j}c^{1}_{ij}\right)\nonumber\\
    \qquad&-&\frac{1}{4}\left(\frac{1}{f^2}\right)(c^{1}_{kj}c^{1}_{kj}),
\end{eqnarray*}
where the Roman indices run from 2, 3, etc.\\
\indent Finally, if $Y$ is a 1-dimensional tangent bundle to a Riemannian foliation with a bundle-like metric, combining Theorem 4.3 and Corollary 4.5, together with the fact that $S_2=0$, we obtain from Formula (3.7)
\begin{eqnarray*}
S&=&S_1\nonumber\\
    \qquad&+&\left(2e_k(c^{1}_{k1})-2c^1_{i1}c^1_{i1}-2c^{1}_{i1}c^j_{ij}-c^i_{1 j}c^{1}_{ij}\right)\nonumber\\
    \qquad&-&\frac{1}{4}\left(\frac{1}{f^2}\right)(c^{1}_{kj}c^{1}_{kj}),
\end{eqnarray*}
where the Roman indices range from 2, 3, etc.
\subsection{Collapsing the metric along the foliation sub-bundle $Y$} In this section, we will discuss the collapse of the metric along the sub-bundle $Y$ and what consequence will that have on the scalar curvature of $M$.
\begin{definition}
A distribution $Y$ of $M$ is called \textbf{everywhere non-involutive} if there exists two vector fields $U,V\in \Gamma(Y)$ so that $[U,V]_x\notin Y_x$ for all $x\in M$. In other words, using our notation above, $Y$ is everywhere non-involutive if and only if $c^i_{\gamma\beta}\neq 0$ for all $i\,,\gamma\,,\beta$.
\end{definition}
\begin{corollary}
If $Y$ is everywhere non-involutive, then for sufficiently large $f$, $S<0$.
\end{corollary}
\begin{proof}
Note that when $Y$ is everywhere non-involutive, then $c^i_{\gamma\beta}\neq 0$. By Formula (3.7), the coefficient of $f^4$ is $-\frac{1}{4}(c^i_{\gamma\beta})^2<0$. Therefore, for sufficiently large $f$, this term will dominate the remaining terms in the formula, making $S<0$ overall.
\end{proof}
\indent This is another way to obtain a metric of negative scalar curvature on a manifold $M$.\\
\indent The condition in Corollary 4.3 serves as an inspiration for us to define a new kind of sub-bundle, that where the metric is nearly positively bundle-like with respect to the sub-bundle.
\begin{definition}
Let $(M,g)$ be a Riemannian manifold. We say that $g$ is \textbf{nearly positively bundle-like (or NPB)} with respect to an involutive vector sub-bundle $Y$ of $M$ if there exists a local adapted orthonormal frame in an open neighborhood around every point $x\in M$ so that $S_2+2e_{\gamma}(c^j_{\gamma j})-c^k_{\alpha k}c^j_{\alpha j}-\frac{1}{2}c^k_{i\beta}c^i_{k\beta}-\frac{1}{2}c^i_{\gamma j}c^i_{\gamma j}>0$, where $S_2$ is the scalar curvature of $M$ restricted to $Y$.
\end{definition}
\indent The reason for such a name is that if $S_2>0$ and $g$ is bundle-like with respect to $Y$, then by the above definition, $g$ is also nearly positively bundle-like with respect to $Y$.\\
\indent Returning to formula (3.7), we have the following corollary.
\begin{theorem}
If the sub-bundle $Y$ is involutive, then the formula in Theorem 14 becomes
\begin{eqnarray*}
S&=&S_1+f^2S_2\\
    \qquad&+&\left(2e_{\gamma}(c^j_{\gamma j})-c^k_{\alpha k}c^j_{\alpha j}-2c^k_{\alpha k}c^{\beta}_{\alpha\beta}-\frac{1}{2}c^k_{i\beta}c^i_{k\beta}-c^i_{\gamma\beta}c^{\gamma}_{i\beta}-\frac{1}{2}c^i_{\gamma j}c^i_{\gamma j}\right)f^2\\
    \qquad&+&\left(2e_k(c^{\beta}_{k\beta})-c^{\gamma}_{i\gamma}c^{\beta}_{i\beta}-2c^{\gamma}_{i\gamma}c^j_{ij}-\frac{1}{2}c^{\alpha}_{\gamma j}c^{\gamma}_{\alpha j}-c^i_{\gamma j}c^{\gamma}_{ij}-\frac{1}{2}c^{\alpha}_{k\beta}c^{\alpha}_{k\beta}\right)\\
    \qquad&-&\frac{1}{4}\left(\frac{1}{f^2}\right)(c^{\alpha}_{kj}c^{\alpha}_{kj}).
\end{eqnarray*}
\indent Specifically, if $g$ is also NPB with respect to $Y$, then for sufficiently large $f$, if $S_2>0$, then $S>0$. In particular, $g'=g_X\oplus \left(\frac{1}{f^2}g_Y\right)$ is a PSC metric on $M$.
\end{theorem}
\indent This is easy to see, since if $g$ is NPB with respect to $Y$, then the coefficient of $f^2$ in the formula above is always positive, and as $f$ increases without bound, then the $f^2$ term dominates, making $S>0$ as a whole.\\
\indent The geometric meaning of the corollary is that if we know that $g$ is nearly positively bundle-like with respect to $Y$, then we can collapse $g$ in the foliation direction so that the total scalar curvature of $M$ will be positive.\\
\indent The following corollary is most likely known to experts.
\begin{corollary}
If a Riemannian foliation on a manifold $M$ has a bundle-like metric such that the leaves have positive scalar curvature, then there exists a PSC metric on $M$.
\end{corollary}
\begin{proof}
\indent If the sub-bundle $Y$ is a tangent bundle of a Riemannian foliation with a bundle-like metric, then Formula (3.7) reduces to
\begin{eqnarray*}
S&=&S_1+f^2S_2\nonumber\\
    \qquad&+&\left(2e_k(c^{\beta}_{k\beta})-c^{\gamma}_{i\gamma}c^{\beta}_{i\beta}-2c^{\gamma}_{i\gamma}c^j_{ij}-\frac{1}{2}c^{\alpha}_{\gamma j}c^{\gamma}_{\alpha j}-c^i_{\gamma j}c^{\gamma}_{ij}-\frac{1}{2}c^{\alpha}_{k\beta}c^{\alpha}_{k\beta}\right)\nonumber\\
    \qquad&-&\frac{1}{4}\left(\frac{1}{f^2}\right)(c^{\alpha}_{kj}c^{\alpha}_{kj}).
\end{eqnarray*}
\indent Since $S_2>0$ by hypothesis, for $f$ sufficiently large, $S>0$ and the corollary is proven.
\end{proof}
\section{Examples and remarks}
In this section, we will provide two examples to illustrate the results presented above. Let's begin with an example to Corollary 4.5.
\subsection{Rescaling along an everywhere non-involutive bundle on the 3-sphere}
 Let's begin with an example to Corollary 4.5. Although $\mathbb{S}^3$ has constant positive sectional curvature, collapsing the sphere metric along an everywhere non-involutive sub-bundle will result in negative scalar curvature for $\mathbb{S}^3$.\\
\indent Let $M=\mathbb{S}^3$ be the unit 3-sphere. In other words, $M=\{(x^1,x^2,x^3,x^4):(x^1)^2+(x^2)^2+(x^3)^2+(x^4)^2=1\}$ with the sphere metric induced from the usual metric in $\mathbb{R}^4$.\\
\indent An adapted orthonormal frame with respect to the above metric is $X=-x^4\partial_1-x^3\partial_2+x^2\partial_3+x^1\partial_4\,,Y=x^3\partial_1-x^4\partial_2-x^1\partial_3+x^2\partial_4\,, Z=-x^2\partial_1+x^1\partial_2-x^4\partial_3+x^3\partial_4$.\\
\indent Furthermore, by a simple calculation, one can see that $[X,Y]=-2Z$, so the sub-bundle of $TM$ spanned by $X$ and $Y$ is non-involutive. Just like in the set up of our problem, we will scale the sphere metric in the part involving $X$ and $Y$ by $\frac{1}{f^2}$, while keeping the rest of the metric unchanged. Hence, the adapted orthonormal frame with respect to the new metric as described above is $\{fX,fY,Z\}=\{X',Y',Z'\}$.\\
\indent Now, also by a simple calculation, we can see that $[X',Y']=-2f^2Z'\,,[X',Z']=2Y'\,,[Y',Z']=-2X$.\\
\indent Denoting the vector fields $X',Y',Z'$ by the numbering 1,2,3 respectively, we can quickly see from the calculations above that the structure functions in this case are
\begin{align*}
    c^1_{12}&=c^2_{12}=0\,,c^3_{12}=-2f^2\\
    c^1_{13}&=c^3_{13}=0\,,c^2_{13}=2\\
    c^1_{23}&=-2\,, c^2_{23}=c^3_{23}=0
\end{align*}
\indent Again, now we will calculate the scalar curvature of $M$ using our formula in Theorem 2.3. In this case, since the structure functions are constants, the first term of the formula above vanishes. Furthermore, the structure functions of the frame have the special property that $c^i_{ij}=0$. Hence, the second terms of the formula above also vanishes.\\
\indent Hence, all in all, the scalar curvature in this case turns to
\begin{align*}
    S&=\sum\limits_{I,J,K=1}^3 \left[-\frac{1}{2}c^I_{KJ}c^K_{IJ}-\frac{1}{4}(c^I_{KJ})^2\right].
\end{align*}
\indent Plugging the structure functions above into the formula, we have
    $$S=-2f^4+8f^2$$
\indent From here, it is easy to see that as in Corollary 4.5, for $f>2$, $S<0$ at all points of $\mathbb{S}^3$. 

\subsection{Rescaling along an NPB foliation on a manifold}
Here is an example of a metric which is nearly bundle-like to a sub-bundle of a manifold, which verifies the existence of such a condition. Note that the foliation described below is not Riemannian, and there does not exist a bundle-like metric for this foliation.\\
\indent Let $N$ be the 3-manifold defined as $\mathbb{R}\times T^2/\mathbb{Z}$, where $T^2=\mathbb{R}^2/\mathbb{Z}^2$ and $m\in\mathbb{Z}$ acts on $\mathbb{R}\times T^2$ by $m(t,x)=(t+m,A^m x)$, where $A=\begin{pmatrix}
1&1\\
0&1
\end{pmatrix}$. Let $\mathbb{S}^2$ be the unit sphere in $\mathbb{R}^3$ with the standard sphere metric. Then, define $M$ to be a 7-manifold so that $M:=N\times \mathbb{S}^2\times H$, where $H$ is a closed surface of constant Gaussian curvature -4. Let $\gamma_t$ be the family of the $t$-parameter curves in $N$, and define the leaves of $M$ as $\gamma_t\times\mathbb{S}^2$.\\
\indent The underlying metric in the $\partial_t\,,\partial_{x_1}\,,\partial_{x_2}\,,\partial_p\,,\partial_q\,,\partial_u\,,\partial_v$ basis, where the first three coordinates belong to the manifold in the example, the next two belong to $\mathbb{S}^2$, and the last two belongs to $H$, is
\begin{align*}
(g_{ij})=\begin{pmatrix}
1&0&0&0&0\\
0&1&-t&0&0\\
0&-t&1+t^2&0&0\\
0&0&0&1&0\\
0&0&0&0&\sin p
\end{pmatrix}\oplus (g_H),
\end{align*}
where $g_H$ is the metric on $H$ that gives rise to constant Gaussian curvature -4.\\
\indent Notice that in this metric, a local adapted orthonormal frame of $TM$ is $\{e_1=\partial_t\,,e_2=\partial_{x_1}\,,e_3=t\,\partial_{x_1}+\partial_{x_2}\,,e_4=\partial_p\,,e_5=\csc p\,\partial_q\,,e_6\,, e_7\}$, where $e_6$ and $e_7$ are vector field in a local adapted orthonormal frame of $TH$. Furthermore, let $\{e_1,e_4,e_5\}$ be the foliation direction of $M$. Using our notation in the paper, therefore, the sub-bundle $Y$ is spanned by the above three vector fields and the other two vector fields span $X$. In the calculations below, we mainly focus on the metric $(g_{ij})$.\\
\indent In the leafwise direction, the metric is the direct sum of the metric on $\mathbb{S}^2$ and the identity metric on $\mathbb{R}$, so it is easy to see that in our case, $S_2=2$. Another calculation reveals that $[e_1,e_3]=e_2$ and all other brackets between a leaf and a non-leaf direction is 0. Hence, only $c^2_{13}=1$, and the other structure functions vanish. Thus, we see that $2e_{\gamma}(c^j_{\gamma j})-c^k_{\alpha k}c^j_{\alpha j}-\frac{1}{2}c^k_{i\beta}c^i_{k\beta}-\frac{1}{2}c^i_{\gamma j}c^i_{\gamma j}=-\frac{1}{2}$ after plugging in and simplifying our expression.\\
\indent In the end, we have $S_2+2e_{\gamma}(c^j_{\gamma j})-c^k_{\alpha k}c^j_{\alpha j}-\frac{1}{2}c^k_{i\beta}c^i_{k\beta}-\frac{1}{2}c^i_{\gamma j}c^i_{\gamma j}=2-\frac{1}{2}=\frac{3}{2}>0$, so indeed $g$ is nearly bundle-like with respect to $Y$. By Theorem 4.5, then, for sufficiently large $f$, the metric
\begin{align*}
(g'_{ij})=\begin{pmatrix}
\frac{1}{f^2}&0&0&0&0\\
0&1&-t&0&0\\
0&-t&1+t^2&0&0\\
0&0&0&\frac{1}{f^2}&0\\
0&0&0&0&\frac{\sin p}{f^2}
\end{pmatrix}\oplus (g_H)
\end{align*}
is a PSC metric. In fact, the scalar curvature of $M$ under this metric can be easily calculated to be $\frac{3}{2}f^2-4$, so as $f$ increases, $S$ will become positive, which verifies Theorem 4.5.\\
\bibliographystyle{amsplain}
\bibliography{reference}

\end{document}